\RequirePackage{fix-cm}
\documentclass[smallextended]{svjour3}       
\smartqed  
\usepackage{lineno,hyperref}
\usepackage{amssymb}
\usepackage{amsmath}
\usepackage{amsfonts}
\usepackage{color}
\usepackage{xcolor}
\usepackage{graphicx}
\usepackage{graphics}
\usepackage{graphicx}
\begin{document}

\title{Lagrangian duality for nonconvex optimization
problems with abstract convex functions
}


\author{Ewa M. Bednarczuk        \and
        Monika Syga 
}


\institute{Ewa M. Bednarczuk  \at Systems Research Institute, Polish Academy of Sciences, Newelska 6, 01–447 Warsaw   \at
	Warsaw University of Technology, Faculty of Mathematics and
	Information Science, ul. Koszykowa 75, 00–662 Warsaw, Poland
              \email{Ewa.Bednarczuk@ibspan.waw.pl}           
           \and
           Monika Syga \at
           	Warsaw University of Technology, Faculty of Mathematics and
           Information Science, ul. Koszykowa 75, 00–662 Warsaw, Poland
           \email{M.Syga@mini.pw.edu.pl}
}

\date{Received: date / Accepted: date}

\maketitle

\begin{abstract}
We  investigate Lagrangian duality   for nonconvex optimization problems. To this aim we use the $\Phi$-convexity theory and  minimax theorem for $\Phi$-convex functions. We provide conditions for zero duality gap and strong duality. Among the classes of functions,  to which our duality results can be applied,  are prox-bounded functions, DC functions, weakly convex functions and paraconvex functions. 

\end{abstract}

\section{Introduction}
\label{intro}
  Lagrangian and conjugate dualities   have far reaching consequences for solution methods and theory in convex  optimization in finite and infinite dimensional spaces. For recent state-of the-art of the topic of convex conjugate duality we refer the reader to the monograph by Radu Bo\c{t} \cite{bot}.

There exist numerous attempts to construct pairs of dual problems in nonconvex optimization e.g., for DC functions \cite{MARTINEZLEGAZ}, \cite{toland}, for composite functions \cite{Bot2003}, DC and composite functions \cite{SunLongLi}, \cite{Sun} and for prox-bounded functions \cite{HareWarPol}.  

 In the present paper we investigate Lagrange duality for general optimization problems within the framework of abstract convexity, namely, within  the theory of $\Phi$-convexity.  The class $\Phi$-convex functions encompasses convex l.s.c. functions, and, among others, the classes of functions mentioned above. A comprehensive study of $\Phi$-convex functions can be found in  the monographs by Pallaschke and Rolewicz \cite{rolewicz} and by Rubinow \cite{rubbook}. 

The set $\Phi$ is a set of functions defined on a given space $X$, called  elementary functions. $\Phi$-convex functions are  
	pointwise suprema of elementary minorizing functions $\varphi\in\Phi$ (e.g. quadratic, quasi-convex).  This  corresponds to  the classical fact that   proper lower semicontinuous
	convex functions are pointwise suprema  of affine  minorizing functions.
	$\Phi$-convexity provides a unifying framework for dealing with important classes of nonconvex functions, e.g, paraconvex (weakly convex), DC, and prox-bounded functions.  In the context of duality theory $\Phi$-convexity is investigated in 
	a large number of papers, e.g.  \cite{ioffe-rub}, \cite{bard}, \cite{dolecki-k},  \cite{Jey2007}, \cite{rolewicz1994}.

	The underlying concepts of $\Phi$-convexity are $\Phi$-conjugation and $\Phi$-subdifferen-\ tiation, which
	 mimic the corresponding constructions of convex analysis,
	i.e., the $\Phi$-conjugate function and the $\Phi$-subdifferential are defined by replacing, in the respective classical definitions,
	the linear (affine) functions with elementary functions $\varphi\in\Phi$ which may not be affine, in general. This motivates the name {\em Convexity without linearity,} coined for $\Phi$-convexity by Rolewicz \cite{rolewicz1994}.



The subject of the present investigation is the Lagrangian duality for optimization problems involving $\Phi$-convex functions. In particular, we investigate Lagrangian duality for problems involving $\Phi_{lsc}$-convex functions (see \eqref{philsc}),
i.e.,  proper lower semicontinuous functions defined on a Hilbert space and minorized by quadratic functions.

The class  of $\Phi_{lsc}$-convex functions embodies many important classes of functions appearing  in optimization, e.g.  prox-bounded functions \cite{RockWets98}, DC (difference of convex) functions \cite{TuyDC}, weakly convex functions \cite{Vial}, paraconvex functions \cite{Rolewiczpara} and lower semicontinuous convex (in the classical sense) functions.

Within the framework of $\Phi$-convexity, the Lagrange duality has been already investigated  on different levels of generality by \cite{burachik}, \cite{gomezvidal}, \cite{penotrub}. 

Main contributions of the paper are as follows. 
\begin{description}
\item [(i)]  Theorem \ref{optgen}  provides necessary and sufficient condition  for zero duality gap for the pair of Lagrange dual problems $L_{P}$ and $L_{D}$ with the Lagrangian ${\mathcal L}$ satisfying some $\Phi$-convexity assumptions, where $\Phi$ is any class of elementary functions.  This condition,  called {\em the intersection property} is defined in Definition \ref{def_2}. See also \cite{Syga2018}.
\item [(ii)] Theorem \ref{opt} and Theorem \ref{new_min_max_2}  provide   conditions  for zero duality gap for the pair of Lagrange dual problems $L_{P}$ and $L_{D}$ with the Lagrangian ${\mathcal L}$ satisfying some $\Phi_{lsc}$-convexity assumptions, where the class $\Phi_{lsc}$ is defined in Example 1.2.
\item  [(iii)]Theorem \ref{par} provides sufficient conditions for the strong duality for problems, where  the optimal value  function $V$ (see the formula \eqref{opvalue})  is paraconvex (weakly convex), and $\text{int\, dom} V\neq\emptyset$.
\item [(iv)] Theorem \ref{opt1} shows that  for  constrained optimization problems with finite optimal values, $\Psi_{lsc}$-convex perturbation functions and $\Phi_{lsc}$-convex Lagrangians, our main condition i.e. intersection property is satisfied. 
\end{description}
 The organization of the paper is as follows. In Section 2 we recall basic concepts of $\Phi$-convexity. We close Section 2 with a  minimax theorem  from  \cite{Syga2018} which is the starting point for our investigations. In Section 3 we introduce the Lagrange function
and we discuss the Lagrangian duality for optimization problems involving
$\Phi$-convex functions (Theorem \ref{optgen}), we also deliver conditions for the strong duality (Theorem \ref{sub}).

In Section 4 and Section 5 we discuss our duality scheme in the class of $\Phi_{conv}$-convex functions nad $\Phi_{lsc}$-convex functions, respectively. In Section 6 we discuss our duality scheme for for a particular form of optimization problem.


\section{$\Phi$-convexity}
\label{preliminaries}

Let $X$ be a set. A function $f:X\rightarrow\hat{\mathbb{R}}:=\mathbb{R}\cup \{-\infty \}\cup \{ +\infty \}$ is proper if its  domain
$\text{dom\,}f=\{x\in X\mid f(x)<+\infty\}\neq\emptyset$ and $f(x)>-\infty$ for any $x\in X$.

Let $\Phi$ be a set of real-valued functions $\varphi:X\rightarrow \mathbb{R}$ and $f:X\rightarrow\hat{\mathbb{R}}$. 
The set
$$
\text{supp}_{\Phi}(f):=\{\varphi\in \Phi\ :\ \varphi\le f\}
$$
is called the {\em support} of $f$ with  respect to $\Phi$, where, 
for any $g,h:X\rightarrow\hat{\mathbb{R}}$,
$g\le h\ \Leftrightarrow\  g(x)\le h(x)\ \ \forall\ x\in X.$
We will use the notation $\text{supp}(f)$ 
whenever the class $\Phi$ is clear from the context. Elements of  class $\Phi$ are called elementary functions.

\begin{definition}(\cite{dolecki-k}, \cite{rolewicz}, \cite{rubbook})
	\label{convf}
	A function $f:X\rightarrow
	\hat{\mathbb{R}}$ is called {\em $\Phi$-convex on $X$} if
	$$
	f(x)=\sup\{\varphi(x)\ :\ \varphi\in\textnormal{supp}(f)\}\ \ \forall\ x\in X.
	$$
	If the set $X$ is clear from the context, we simply say that $f$ is {\em $\Phi$-convex}. A function $f:X\rightarrow
	\hat{\mathbb{R}}$ is called {\em $\Phi$-convex at $x_{0}\in X$} if
	$$
	f(x_{0})=\sup\{\varphi(x_{0})\ :\ \varphi\in\textnormal{supp}(f)\}.
	$$
\end{definition}
We have $\textnormal{supp}(f)=\emptyset$ iff $f\equiv-\infty$. By convention, we consider that the function $f\equiv-\infty$ is  $\Phi$-convex (c.f. \cite{rubbook}).
A $\Phi$-convex function $f:X\rightarrow\hat{\mathbb{R}}$ is proper if $\textnormal{supp}(f)\neq\emptyset$ and the  domain of $f$ is nonempty, i.e.
$$
\text{dom}(f):=\{x\in X \ : \ f(x)<+\infty \}\neq \emptyset.
$$
If $X$ is a topological space and a given class $\Phi$ consists of functions $\varphi:X\rightarrow\mathbb{R}$  which are lower semicontinuous on $X$, then $\Phi$-convex functions are lower semicontinuous on $X$ (\cite{zalinescu2002}).
Note that $\Phi$-convex functions defined above may admit the value $+\infty$  which allows us to consider  indicator functions within the framework of $\Phi$-convexity. 

Analogously,  we say that $f:X\rightarrow\bar{\mathbb{R}}$ is 
$\Phi$-concave on $X$ if
$$
f(x)=\inf\{\varphi(x)\ : \ \varphi\in\Phi, \ f\le\varphi\}.
$$
Clearly, $f$ is $\Phi$-concave on $X$ iff $-f$ is $-\Phi$-convex on $X$.

The following classes of elementary functions are considered.

\begin{example}\label{exampleone}\mbox{}
	
	\begin{enumerate}
	\item $$
	\Phi_{conv}:=\{\varphi : X \rightarrow \mathbb{R},\ \varphi(x)= \left\langle \ell,x\right\rangle+c, \ \ x\in X,\  \ell\in X^{*}, \ c\in \mathbb{R}\}, 
	$$
	where $X$ is a topological vector  space, $X^{*}$ is the dual space to $X$. It is a well known fact
	(see for example Proposition 3.1 of \cite{Ekeland}) that a proper convex lower semicontinuous function $f:X\rightarrow \bar{\mathbb{R}}$ is  $\Phi_{conv}$-convex. For the analysis of the class $\Phi$ which generates all convex functions \cite{Syga2018}.
\item 
\begin{equation} 
\label{philsc}
\Phi_{lsc}:= \{\varphi : X \rightarrow \mathbb{R}, \ \varphi(x)=-a\|x\|^2+ \left\langle \ell,x\right\rangle+c, \ \ x\in X,\  \ell\in X^{*}, \ a\geq 0, \ c\in \mathbb{R} \},
\end{equation}
where $X$ is a normed space. If $X$ is a Hilbert space, then $f: X\rightarrow  \bar{\mathbb{R}}$ is $\Phi_{lsc}$-convex iff $f$ is lower semicontinuous and minorized by a quadratic function $q(x):=-a\|x\|^{2}-c$
on $X$ (e.g. \cite{rubbook}, Example 6.2). The class of $\Phi_{lsc}$- convex functions encompasses:  prox-bounded functions \cite{poli-rock96} and weakly convex functions  \cite{Vial}, known also under the name paraconvex functions \cite{Rolewiczpara} and semiconvex functions \cite{cannarsa}. 
\item 
$$
\Phi^{+}_{lsc}:= \{\varphi : X \rightarrow \mathbb{R}, \ \varphi(x)=a\|x\|^2+ \left\langle \ell,x\right\rangle+c, \ \ x\in X,\  \ell\in X^{*}, \ a\geq 0, \ c\in \mathbb{R} \},
$$ 
where $X$ is a normed space. If $X$ is a Hilbert space, then $f: X\rightarrow  \bar{\mathbb{R}}$ is $\Phi^{+}_{lsc}$-concave iff $f$ is upper semicontinuous and majorized by a quadratic function $q(x):=a\|x\|^{2}-c$
on $X$ (\cite{rubbook}, Example 6.2). 
\end{enumerate}
\end{example}

\subsection{$\Phi$-Subgradients}

\begin{definition}
	\label{def_subgradient_gen}
	An element  $\varphi\in\Phi$ is called a {\em $\Phi$-subgradient} of a  function  $f:X\rightarrow \bar{\mathbb{R}}$ at $\bar{x}\in\text{dom\,} f$, if the following inequality holds
	\begin{equation}
	\label{phi_sub}
	f(x)-f(\bar{x}) \geq \varphi(x)-\varphi(\bar{x}) \ \ \forall \ x\in X .\end{equation}
	The set of all  $\Phi$-subgradients of $f$ at  $\bar{x}$ is denoted as $\partial_{\Phi}  f(\bar{x})$.
	
		Let $\varepsilon>0$. An element  $\varphi\in\Phi$ is called a {\em $\Phi-\varepsilon$-subgradient} of a  function  $f:X\rightarrow \bar{\mathbb{R}}$ at $\bar{x}\in\text{dom\,} f$, if the following inequality holds
	\begin{equation}
	\label{phi_sub1}
	f(x)-f(\bar{x}) \geq \varphi(x)-\varphi(\bar{x})-\varepsilon \ \ \forall \ x\in X .
	\end{equation}
	The set of all  $\Phi-\varepsilon$-subgradients of $f$ at  $\bar{x}$ is denoted as $\partial^{\varepsilon}_{\Phi}  f(\bar{x})$.
	\end{definition}

Definition \ref{def_subgradient_gen} was introduced in \cite{rolewicz} and \cite{rubbook}.

\begin{definition}
	\label{def_subgradient}
	An element  $(a,v)\in \mathbb{R}_+\times X$ is called a {\em $\Phi_{lsc}$-subgradient} of a  function  $f:X\rightarrow \bar{\mathbb{R}}$ at $\bar{x}\in\text{dom\,} f$, if the following inequality holds
	\begin{equation}
	\label{phinew}
	f(x)-f(\bar{x}) \geq 	\langle v, x-\bar{x}\rangle -a\| x\|^2+a\|\bar{x}\|^2,  \  \ \ \ \ \ \forall \ x\in X .
	\end{equation}
	The set of all  $\Phi_{lsc}$-subgradients of $f$ at  $\bar{x}$ is denoted as $\partial_{lsc}  f(\bar{x})$.
	
		Let $\varepsilon>0$. An element  $(a,v)\in \mathbb{R}_+\times X$ is called a {\em $\Phi_{lsc}-\varepsilon$-subgradient} of a  function  $f:X\rightarrow \bar{\mathbb{R}}$ at $\bar{x}\in\text{dom\,} f$, if the following inequality holds
	\begin{equation}
	\label{phinew1}
	f(x)-f(\bar{x}) \geq 	\langle v, x-\bar{x}\rangle -a\| x\|^2+a\|\bar{x}\|^2-\varepsilon,  \  \ \ \ \ \ \forall \ x\in X .
	\end{equation}
	The set of all  $\Phi_{lsc}-\varepsilon$-subgradients of $f$ at  $\bar{x}$ is denoted as $\partial_{lsc}^{\varepsilon}  f(\bar{x})$
\end{definition}

\begin{remark}
	\label{rem1}
	It is easy to show that  $\partial_{lsc}  f(\bar{x})$ and  $\partial^{\varepsilon}_{lsc}  f(\bar{x})$, $\varepsilon>0$, are convex sets for all $\bar{x}\in\text{dom\,} f$.
	
	Moreover,  $\partial^{\varepsilon}_{\Phi}  f(\bar{x})\neq\emptyset$ for any $\bar{x}\in\text{dom\,} f.$ Indeed, if
	$\bar{x}\in\text{dom\,} f$, then, by the $\Phi$-convexity of $f$, we have  $f(\bar{x})=\sup\{\varphi(\bar{x}),\ \varphi\in \text{supp\,} f\}$. By the definition of supremum,  for any $\varepsilon>0$, there exists
	$\bar{\varphi}\in\text{supp\,} f$ such that
	$$
	\bar{\varphi}(\bar{x})> f(\bar{x})-\varepsilon
	$$
	and consequently,
	$f(x)-f(\bar{x})\ge \bar{\varphi}(x)-\bar{\varphi}(\bar{x})-\varepsilon$, i.e. $\bar{\varphi}\in\partial^{\varepsilon}_{\Phi}  f(\bar{x})$.
\end{remark}

\subsection{$\Phi$-conjugation}
\label{conjugation}

Let $f:X\rightarrow\bar{\mathbb{R}}$. The function $f^{*}_{\Phi}:\Phi\rightarrow\bar{\mathbb{R}}$,
\begin{equation}
\label{conjugate} 
f^{*}_{\Phi}(\varphi):=\sup_{x\in X} (\varphi(x)-f(x))
\end{equation}
is called the {\em $\Phi$-conjugate} of $f$. The function $f^{*}_{\Phi}$ is $\Phi$-convex (c.f.Proposition 1.2.3 of \cite{rolewicz}). Accordingly, the {\em second $\Phi$-conjugate} of $f$ is defined as
$$
f^{**}_{\Phi}(x):=\sup_{\varphi\in\Phi}(\varphi(x)-f^{*}_{\Phi}(\varphi)).
$$

\begin{theorem} (\cite{rolewicz} and Theorem 1.2.6, \cite{rubbook}, Theorem 7.1) 
	\label{conju}
	Function $f:X\rightarrow\hat{\mathbb{R}}$ is $\Phi$-convex if and only if 
		$$
		f(x)=f^{**}_{\Phi}(x) \ \ \ \ \ \ \ \ \ \ \ \forall \ \ x\in X.
		$$	
\end{theorem}
For any function $f:X\rightarrow\hat{\mathbb{R}}$   the Young inequality holds.  i.e.
$$
f(x)+f^*_{\Phi}(x)\geq \varphi(x), \ \ \forall \ \ \varphi\in\Phi, \ x\in X.
$$
Moreover we have the following proposition.
\begin{proposition}(\cite{rubbook}, Proposition 7.7, \cite{rolewicz}, Proposition 1.2.4)
	\label{you}
	Let $\varphi \in \Phi$ and $x\in X$. The following conditions are equivalent
	\begin{description}
		\item{(i)} $
		f(x)+f^*_{\Phi}(x)=\varphi(x)
		$
		\item{(ii)} $\varphi \in \partial_{\Phi}f(x)$.
		\end{description}
	
	\end{proposition}
\subsection{Minimax Theorem for $\Phi$-convex functions}
Minimax theorems for function $a:X\times Y\rightarrow\hat{\mathbb{R}}$ such that for each $y\in Y$ the function $a(\cdot,y):X\rightarrow\hat{\mathbb{R}}$ is $\Phi$-convex  are based on the following  property introduced in \cite{bed-syg} and investigated in \cite{Syga2018}, \cite{sygamoor}.
\begin{definition}
	\label{def_2}
	Let $\varphi_{1},\varphi_{2}:X\rightarrow\mathbb{R}$ be any two functions and $\alpha\in\mathbb{R}$. We say that $\varphi_{1}$ and $\varphi_{2}$ have
	{\em the intersection property on $X$ at the level $\alpha\in \mathbb{R}$} 
	iff  for every $t\in [0,1]$
	{\small 
		\begin{equation}
		\label{eq-n}
		\begin{array}[t]{c}
		[t\varphi_{1}+(1-t)\varphi_{2}<\alpha]\cap [\varphi_{1}<\alpha]=\emptyset \ \ \ \text{or}
		\ \ \
		[t\varphi_{1}+(1-t)\varphi_{2}<\alpha]\cap [\varphi_{2}<\alpha]=\emptyset,
		\end{array}
		\end{equation}}
		where $[\varphi<\alpha]:=\{x\in X :\ \varphi(x)<\alpha \}$.
\end{definition}
 Let us note that, by Proposition 4 of \cite{bed-syg}, in the class $\Phi_{lsc}$ the condition \eqref{eq-n} takes the form
\begin{equation}
\label{eq-n1}
    [\varphi_{1}<\alpha]\cap[\varphi_{2}<\alpha] =\emptyset.
\end{equation}

The following  minimax theorem was proved in \cite{Syga2018}.
\begin{theorem} (\cite{Syga2018}, Theorem 5.1)
	\label{new_min_max}
	Let $X$ be a real vector space and $Z$ be a convex subset of a real vector space U. Let $a:X\times Z\rightarrow\hat{\mathbb{R}}$ be a function such that for any $z\in Z$ the  function $a(\cdot,z):X\rightarrow\hat{\mathbb{R}}$ 
	is $\Phi$-convex on $X$ and for any $x\in X$ the function $a(x,\cdot):Z\rightarrow\hat{\mathbb{R}}$   is concave on $Z$. 
	
The following conditions are equivalent:
\begin{description}
	\item [{\em (i)}] for every $\alpha\in\mathbb{R}$, $\alpha < \inf\limits_{x\in X} \sup\limits_{z\in Z} a(x,z)$, there exist $z_{1}, z_{2}\in Y$ and $\varphi_{1}\in \textnormal{supp } a(\cdot, z_{1})$, $\varphi_{2}\in \textnormal{supp } a(\cdot, z_{2})$ such that the intersection property holds for $\varphi_{1}$ and $\varphi_{2}$ on $X$ at the level $\alpha$,
	\item [{\em (ii)}] $\sup\limits_{z\in Z} \inf\limits_{x\in X} a(x,z)=\inf\limits_{x\in X} \sup\limits_{z\in Z} a(x,z).$
\end{description}
\end{theorem}

In the class $\Phi_{lsc}$ the 
intersection property is related to the following zero subgradient condition (see \cite{sygamoor}).

\begin{definition}
	\label{zsc}
	Let $f,g:X\rightarrow\hat{\mathbb{R}}$ be $\Phi_{lsc}$-convex functions. We say that $f$ and $g$
	satisfy the {\em zero subgradient condition} at $(x_1,x_2)$, where $x_{1}\in\text{dom}(f)$, $x_{2}\in\text{dom}(g)$  if 
	\begin{equation*}
	\label{main-con}
	0\in\mbox{co}(\partial_{lsc}f(x_1)\cup\partial_{lsc} g(x_2)),
	\end{equation*}
	where $\text{co}(\cdot)$ is the standard convex hull of a set. If $x_1=x_2=\bar{x}$ we say that  $f$ and $g$
	satisfy the {\em zero subgradient condition} at $\bar{x}$.
\end{definition}
The following proposition was proved in \cite{sygamoor} (Proposition 12), for convenience of the reader we rewrite the proof with $\varepsilon=0$. 
\begin{proposition}
\label{subzero}
	Let $X$ be a Hilbert space, $f,g:X\rightarrow\hat{\mathbb{R}}$ be $\Phi_{lsc}$-convex functions, $\alpha\in\mathbb{R}$. Assume that $\bar{x}\in \text{dom}(f)\cap \text{dom}(g)$ and $\bar{x}\in [f\geq \alpha]\cap[g\geq \alpha]$.

If $f$ and $g$ satisfy the zero subgradient condition at $\bar{x}$ then, there exist
$  \varphi_{1}\in\text{supp} (f), \ \varphi_{2}\in\text{supp} (g)$ for
which the intersection property holds at the level  $ \alpha$.
\end{proposition}
\begin{proof}
	By Remark \ref{rem1}, we only need to consider the case where $(a_1,v_1)\in \partial_{lsc} f(\bar{x}) $ and $(a_2,v_2)\in\partial_{lsc} g(\bar{x}) $ are such that $\lambda a_1+\mu a_2=0$,  $\lambda v_1+\mu v_2=0$, for some $\lambda,\mu\ge 0$, $\lambda +\mu=1$ .
		
		If $\lambda =0$, then $0\in \partial_{lsc} g(\bar{x})$ and  consequently, $g(x)\geq g(\bar{x})$ for all $x\in X$. By assumption, $g(\bar{x})\geq \alpha$, so  $\varphi_1 \equiv \alpha$ is in the support of $g$, $\varphi_1\in \text{supp} (g)$, hence, $\varphi_1$ and any function $\varphi_2\in \text{supp}(f)$ have the intersection property at the level $\alpha$, since $[\varphi_1<\alpha]=\emptyset$.
		By analogous reasoning, we get the desired conclusion if $\mu=0$.
		
		Now assume that  $\lambda>0$ and $\mu> 0$. This implies that $a_1=a_2=0$, since $a_1,a_2\geq 0$. Let us take $\varphi_1\in \text{supp} (f)$ and $\varphi_2\in \text{supp} (g)$.
		$$
		\varphi_1(x):=\langle v_1, x- \bar{x}\rangle +f(\bar{x})\ \ \ \ \mbox{and}\ \ \ \
		\varphi_2(x):=\langle v_2, x- \bar{x}\rangle +g(\bar{x}) 
		$$ 
		for all $x\in X$.  We show that $\varphi_1$ and $\varphi_2$ have the intersection property at the level $\alpha$. Let $x_1\in [\varphi_1<\alpha]$, we have
		$$
		\begin{array}{l}
		\varphi_1(x_1)<\alpha \ \ \Leftrightarrow \\
		\langle v_1, x_1- \bar{x}\rangle +f(\bar{x})<\alpha  \ \  \Leftrightarrow \\
		\langle v_1, x_1- \bar{x}\rangle <\alpha-f(\bar{x})
		\end{array}
		$$
		By assumption that $\bar{x}\in [f\geq \alpha]\cap [g\geq \alpha] $, we have $ \alpha-f(\bar{x})\leq 0$, so
		$$
		\langle v_1, x_1- \bar{x}\rangle <0.
		$$
		Since, $x_1$ is chosen arbitrarily, we get   $ [\varphi_1<\alpha]\subset [\langle v_1, \cdot- \bar{x}\rangle<0 ]$. By similar calculations we get  $ [\varphi_2<\alpha]\subset [\langle v_2, \cdot- \bar{x}\rangle<0 ]$. 
		
		Since $\lambda v_1+\mu v_2=0$ we have
		$$
		[\langle v_2, \cdot- \bar{x}\rangle<0 ]=[-\langle v_1, \cdot- \bar{x}\rangle<0 ]=[\langle v_1, \cdot- \bar{x}\rangle>0 ]
		$$
		and consequently
		$$
		[\varphi_1<\alpha] \cap [\varphi_2<\alpha]=\emptyset
		$$
		which completes the proof.
	\end{proof}
\vskip 0.2 true in 

\section{ Duality for $\Phi$-convex Lagrangian.}
\label{section_2}


Let $X$ be a vector space and let $\Phi$ be a class of elementary  functions $\varphi:X\rightarrow \mathbb{R}$.
In this section we investigate Lagrange duality  for   minimization
problem 
\begin{equation}
\label{problem}
\tag{P}
\text{Min}\ \ \ \ f(x) \ \ \ \ \ \ x\in X,
\end{equation}
where  $f:X\rightarrow \bar{\mathbb{R}}$ is a proper 
function. 

We introduce the Lagrange function for  problem $(P)$ by applying {\em perturbation/parametrization approach}, see e.g. \cite{balder}, \cite{Bonnans},  \cite{toland}.

Let $Y$ be a vector space.  
A   function $p:X\times Y\rightarrow \hat{\mathbb{R}}$  satisfying 
$$
p(x,y_{0})=f(x),  \ \ \ \text{for some} \ \ \ y_0\in Y
$$
is called {\em a perturbation function to problem $(P)$}.  Often we take $y_{0}=0$. Clearly, $\text{dom\,} p\neq\emptyset$ since $\text{dom\,} f\neq\emptyset$ but, in general, $p$ need not to be proper if $f$ is  a proper.

By using  function $p(\cdot,\cdot)$, the  family of parametric problems $(P_{y})$ is defined as
\begin{equation}
\label{problemy}
\tag{$P_y$}
\text{Min}\ \ \ \ p(x,y),\ \ \ \ x\in X,
\end{equation}
where $(P)$ coincides with $(P_{y_0})$. 

Let $\Psi$ be a class of elementary functions $\psi:Y\rightarrow \mathbb{R}$.  
We consider the  Lagrangian ${\mathcal L}:X\times\Psi:\rightarrow \bar{\mathbb{R}}$  
defined as
\begin{equation}
\label{genlag}
{\mathcal L}(x, \psi):=\psi(y_0)-p^{*}_{x}(\psi),
\end{equation}
where $p_{x}^{*}:\Psi\rightarrow \bar{\mathbb{R}}$, and 
$$
p^{*}_{x}(\psi)=\sup\limits_{y\in Y}\{\psi(y)-p(x,y)\}
$$ 
is the $\Psi$-conjugate of the function $p_{x}(\cdot):=p(x,\cdot)$, $x\in X$.  When $Z_{1}:=Y$, $Z_{2}:=\Psi$, $c(y,\psi):=\psi(y)$, and $y_{0}=0$, Lagrangian defined by \eqref{genlag} coincides with Lagrangian $L(x,y)$ as defined in  Proposition 1 of \cite{penotrub}. Also the Lagrangian given by formula 2.1 of \cite{BurachikRubinov} coincides with \eqref{genlag}.

 Analogous definitions of Lagrangian have been introduced  in convex case e.g.  \cite{Bonnans}, in DC case  \cite{toland} and in general $\Phi$-convex case \cite{BurachikRubinov}, \cite{dolecki-k}, \cite{Kurcyusz} and \cite{rolewicz}, Section 1.7. 

Consider the  Lagrangian primal problem 
\begin{equation}
\label{lprob}
\tag{$L_P$}
val(L_{P}):=\inf_{x\in X}\sup_{\psi\in\Psi}{\mathcal L}(x, \psi).
\end{equation}
\begin{proposition}
	\label{primal}
	Problems \eqref{problem} and \eqref{lprob} are equivalent in sense that 
	$$\inf\limits_{x\in X}f(x)=\inf\limits_{x\in X}\sup\limits_{\psi\in\Psi}{\cal L}(x,\psi),
	$$
	if and only if  $p(x,\cdot)$ is $\Psi$-convex function at $y_{0}\in Y$ for all $x\in X$. 
\end{proposition}
\begin{proof}
	It is enough to observe that the following equality holds
	$$
	\sup_{\psi\in\Psi}{\mathcal L}(x, \psi)= \sup_{\psi\in\Psi}\{\psi(y_0)-p^{*}_{x}(\psi)\}=p^{**}_{x}(y_0)=p(x,y_0).
	$$
	where the latter equality follows from Theorem \ref{conju}.
\end{proof}
The dual problem to $\eqref{lprob}$ is defined as
\begin{equation}
\label{ldual}
\tag{$L_D$}
val(L_{D} ):=\sup_{\psi\in\Psi}\inf_{x\in X}{\mathcal L}(x, \psi).
\end{equation}
Problem \eqref{ldual} is called the Lagrangian dual.  The inequality
\begin{equation} 
\label{minimaxineq}
val(L_{D})\le val(L_{P})
\end{equation}
always holds. 
We say that 
the {\em zero duality gap} holds for problems \eqref{lprob} and \eqref{ldual} if the equality $val(L_{P})=val(L_{D})$ holds.

The following theorem provides sufficient and necessary conditions for the zero duality gap, 
within the framework of  $\Phi$-convexity, where $\Phi$ is any class of elementary functions. Up to our knowledge, this is the first result on this level of generality.

We  start with some preliminary observations. We have
$$
\text{dom\,} {\mathcal L}:\begin{array}[t]{l}
=\{(x,\psi)\in X\times \Psi\mid {\mathcal L}(x,\psi)<+\infty\}\\
=\{(x,\psi)\in X\times \Psi\mid \inf\limits_{y\in Y}\{p_{x}(y)-\psi(y)\}<+\infty\}.
\end{array}
$$
Observe that $\text{dom\,} p_{x}\neq\emptyset$ for $x\in\text{dom\,} f$
and, for any $\psi\in\Psi$,
$$
p^{*}_{x}(\psi)\ge \psi(y_{0})-p(x,y_{0})=\psi(y_{0})-f(x),
$$
 i.e. $p^{*}_{x}(\cdot)>-\infty$ for any $x\in\text{dom\,}f$. Since $f$ is  proper,  $\text{dom\,}f\neq\emptyset$  and
$\text{dom\,}f\subset\text{dom\,}{\mathcal L}(\cdot,\psi)$ 
 for any $\psi\in\Psi$.

On the other hand, under the assumptions of Proposition \ref{primal}, since $f$ is proper we have $f(x)=\sup\limits_{\psi\in\Psi}{\mathcal L}(x,\psi)>-\infty$ for any $x\in X$, i.e. ${\mathcal L}(x,\psi)>-\infty$ for some $\psi\in\Psi$ which means that among functions ${\mathcal L}(\cdot, \psi)$, $\psi\in\Psi$ may exist proper functions. In conclusion,
\begin{equation} 
\label{proper}
\text{dom\,}{\mathcal L}(\cdot,\psi) \neq\emptyset\text{  for every  } \psi\in\Psi\ \ \text{and}\ \text{supp\,} {\mathcal L}(\cdot,\psi)\neq\emptyset\ \ \ \text{for some  } \psi\in\Psi.
\end{equation}
 Moreover, the following fact holds. 
 \begin{equation}
 \label{observation}
 \exists_{\bar{x}\in X}\ \exists_{\bar{\psi}\in\Psi} {\mathcal L}(\bar{x},\bar{\psi})=+\infty \ \Leftrightarrow\ \forall_ {\psi\in\Psi}\ {\mathcal L}(\bar{x},\psi)=+\infty.
 \end{equation}

 To see this it is enough to note that the condition
 ${\mathcal L}(\bar{x},\bar{\psi})=\bar{\psi}(y_{0})-p^{*}_{\bar{x}}(\bar{\psi})=+\infty$ for some $\bar{x}\in X$ and $\bar{\psi}\in\Psi$ can be  rewritten as
 $\inf_{y\in Y}\{p_{\bar{x}}(y)-\bar{\psi}(y)\}=+\infty$ 
which means that $p_{\bar{x}}(y)=p(\bar{x},y)=+\infty$ for any $y\in Y$ i.e.
 $\text{dom\,} p_{\bar{x}}=\emptyset$ and consequently $\forall_ {\psi\in\Psi}\ {\mathcal L}(\bar{x},\psi)=+\infty$.


\begin{theorem}
	\label{optgen}
	Let $X, U$ be  vector spaces. Let $\Psi\subset U$ be a convex set of elementary functions $\psi:Y\rightarrow \mathbb{R}$ and let the  function ${\mathcal L}(\cdot,\psi):X\rightarrow\hat{\mathbb{R}}$, given by \eqref{genlag},
	be $\Phi$-convex on $X$   for any $\psi\in\Psi$. Assume that $p(x,\cdot)$ is $\Psi$-convex function at $y_{0}\in Y$ for all $x\in X$. 
	The following are equivalent:
	\begin{description}
\item[(i)]	for every $\alpha <\inf_{x\in X}\sup_{\psi\in\Psi}{\mathcal L}(x, \psi)$ there exist $\psi_1,\psi_2\in \Psi$ and  $ \varphi_1 \in\text{supp} {\mathcal L}(\cdot,\psi_1)$  and $ \varphi_2 \in\text{supp} {\mathcal L}(\cdot,\psi_2)$ such that functions $\varphi_1$ and $\varphi_2$ have the intersection property at the level $\alpha$;
	
\item[(ii)]	$$\inf\limits_{x\in X} f(x)=\inf_{x\in X}\sup_{\psi\in\Psi}{\mathcal L}(x, \psi)=\sup_{\psi\in\Psi}\inf_{x\in X}{\mathcal L}(x, \psi).$$
\end{description}
\end{theorem}
\begin{proof}

The first equality follows from Proposition \ref{primal}. To prove the second equality we apply Theorem  \ref{new_min_max}, i.e. we need to show that 	${\mathcal L}(x,\cdot)$ is a concave function of $\psi$ for all $x\in X$, i.e.

\begin{equation} 
\label{concavity} 
{\mathcal L}(x, t\psi_1+(1-t)\psi_2)\ge t{\mathcal L}(x, \psi_1)+(1-t){\mathcal L}(x, \psi_2)
\end{equation}
for any $t\in[0,1]$, $x\in X$ and any $\psi_{1},\psi_{2}\in\Psi$.

 Take any  $\psi_{1}, \psi_{2}\in\Psi$. 
${\mathcal L}(x,\psi_{1})$ and ${\mathcal L}(x,\psi_{2})$ are finite.
Let  $t\in [0,1]$. We have
$$
	\begin{array}{l}
{\mathcal L}(x, t\psi_1+(1-t)\psi_2)=\\
=t\psi_1(y_0)+(1-t)\psi_2(y_0)-\sup\limits_{y\in Y}\{t\psi_1(y)+(1-t)\psi_2(y)-p(x,y)\}\\
=t\psi_1(y_0)+(1-t)\psi_2(y_0)-\sup\limits_{y\in Y}\{t\psi_1(y)+(1-t)\psi_2(y)-tp(x,y)-(1-t)p(x,y)\}\\
\geq t\psi_1(y_0)+(1-t)\psi_2(y_0) -t\sup\limits_{y\in Y}\{\psi_1(y)-p(x,y)\}-(1-t)\sup\limits_{y\in Y}\{\psi_2(y)-p(x,y)\}\\
=t{\mathcal L}(x, \psi_1)+(1-t){\mathcal L}(x, \psi_2).
\end{array}
$$


Hence, ${\mathcal L}(x,\cdot)$ is concave on $\Psi$. Observe that ${\mathcal L}(x,\psi_{1})$ and ${\mathcal L}(x,\psi_{2})$ can take infinite values.
	\end{proof}

\begin{remark}
    \label{remarkinfinite} 
    Assume that  $p(x,\cdot)$ is $\Psi$-convex function at $y_{0}\in Y$ for all $X$. By Proposition \ref{primal}, $f(x)=\sup_{\psi\in\Psi}{\mathcal L}(x,\psi)$. Since  $f$ is proper, it may only happen that
    $$
    \inf\limits_{x\in X}f(x)=\inf\limits_{x\in X}\sup\limits_{\psi\in\Psi}{\mathcal L}(x,\psi)=-\infty\ \ \text{or   }
     \inf\limits_{x\in X}f(x)=\inf\limits_{x\in X}\sup\limits_{\psi\in\Psi}{\mathcal L}(x,\psi)\ \ \text{is finite}.
     $$ 
     In the first case, when $val(L_{P})=-\infty$, in view of \eqref{minimaxineq}, there is nothing to prove and the condition $(i)$ of Theorem \ref{optgen} is authomatically satisfied.
\end{remark}

 Theorem \ref{optgen} can also be  formulated    
		 in terms of the optimal value function of the problems $(P_{y})$. The optimal value function of $(P_{y})$,
		$V:Y\rightarrow\bar{\mathbb{R}}$ is defined as 
		\begin{equation}
		\label{opvalue}
		    	V(y):=\inf_{x\in X}p(x,y).
		\end{equation}
	
	In Proposition \ref{propoptvalue} below the convexity of elementary function set $\Psi$  is not required.
		\begin{proposition}
		\label{propoptvalue}
Assume that  $p_{x}=p(x,\cdot)$ is $\Psi$-convex function on $Y$ for all $x \in X$.	The following are equivalent:
\begin{description}
	\item[(i)]  	$$\inf_{x\in X}f(x)=\inf_{x\in X}\sup_{\psi\in\Psi}{\mathcal L}(x, \psi)=\sup_{\psi\in\Psi}\inf_{x\in X}{\mathcal L}(x, \psi).$$
	\item[(ii)] $$
	V(y_0)=V^{**}(y_0),$$
	where $y_0\in Y$ and $p(x,y_{0})=f(x)$, for any $x\in X$.
\end{description}		
			\end{proposition}
			\begin{proof}
			 	For any $\psi\in \Psi$, we have
	\begin{equation}
	\label{con2}
		V^*(\psi)=\sup_{y\in Y}\{\psi(y)-  \inf_{x\in X}p(x,y)\}=\sup_{y\in Y}\sup_{x\in X}\{\psi(y)-  p(x,y)\}=\sup_{x\in X}p^*_{x}(\psi).
		\end{equation}
		On the other hand, for the dual  function we have
		$$
		\inf_{x\in X}L(x,\psi)=	\inf_{x\in X}\{ \psi(y_0)-p^*_{x}(\psi)\}= \psi(y_0)-\sup_{x\in X}\{p^*_{x}(\psi)\}.
		$$
		By \eqref{con2},
		$$
			\inf_{x\in X}L(x,\psi)=\psi(y_0)-V^*(\psi),
		$$
		and for  the dual \eqref{ldual} we have
		\begin{equation} 
		\label{dualvalue}
		\sup_{\psi\in\Psi}\inf_{x\in X}{\mathcal L}(x, \psi)=	\sup_{\psi\in\Psi}\{\psi(y_0)-V^*(\psi)\} =V^{**}(y_0).
		\end{equation}
		By Proposition \ref{primal}, when the perturbation function $p(x,\cdot)$ is $\Psi$-convex for each $x\in X$ then
		$$
		\inf_{x\in X}\sup_{\psi\in\Psi}{\mathcal L}(x, \psi)=V(y_0)
		$$
		which completes the proof.
			\end{proof}
			
			In reflexive Banach spaces for some particular coupling functions, conditions ensuring $(ii)$ were proved in  Theorem 4.1 and Proposition 4.1 of \cite{BurachikRubinov}.
				\begin{corollary}
					Assume that  $p_{x}=p(x,\cdot)$ is $\Psi$-convex function on $Y$ for all $x \in X$.	The following are equivalent:
					\begin{description}
						\item[(i)]  	$$\inf_{x\in X}f(x)=\inf_{x\in X}\sup_{\psi\in\Psi}{\mathcal L}(x, \psi)=\sup_{\psi\in\Psi}\inf_{x\in X}{\mathcal L}(x, \psi).$$
						\item[(ii)]  $V$ is $\Psi$-convex at $y_0$ i.e. $V(y_0)=\sup\{\psi(y_0), \ \psi\in\text{supp}\, V(y) \}$  (cf. Theorem \ref{conju}).
					\end{description}		
				\end{corollary}
				\begin{proof}
				 Follows directly from Theorem 	\ref{conju} and Proposition \ref{propoptvalue}.
				\end{proof}
				The following corollary shows the relationship between the $\Phi$-convexity of the Lagrangian and the $\Psi$-convexity of the perturbation function. 
			\begin{corollary}
			Let $\Psi$ be a convex set of elementary functions $\psi:Y\rightarrow \mathbb{R}$. Assume that  $p_{x}=p(x,\cdot)$ is $\Psi$-convex function on $Y$ for all $x \in X$.
			Assume that for any $\psi\in\Psi$ the  function ${\mathcal L}(\cdot,\psi):X\rightarrow\hat{\mathbb{R}}$,  defined by \eqref{genlag},
				is $\Phi$-convex on $X$.  The following are equivalent.
			\begin{description}	
				\item [(i)] For every $\alpha <\inf\limits_{x\in X}\sup\limits_{\psi\in\Psi}{\mathcal L}(x, \psi)$ there exist $\psi_1,\psi_2\in \Psi$ and  $ \varphi_1 \in\text{supp}\, {\mathcal L}(\cdot,\psi_1)$  and $ \varphi_2 \in\text{supp}\, {\mathcal L}(\cdot,\psi_2)$ such that functions $\varphi_1$ and $\varphi_2$ have the intersection property at the level $\alpha$
				\item [(ii)] The optimal value function  $V$ is $\Psi$-convex at  $y_0$.
				\end{description}
				\end{corollary}
				\begin{proof} 
				Follows from Theorem 	\ref{optgen} and Proposition \ref{propoptvalue}.   To  see the implication $(i)$ $\Rightarrow$ $(ii)$ we need to show the following. 
				\begin{description} 
				\item[(a)] If $V(y_{0})=\inf\limits_{x\in X} f(x)=-\infty$, then $V(y)=-\infty$ for all $y\in Y$.
				\item[(b)] If $V(y_{0})$ is finite, then $\text{supp\,}\neq\emptyset$ and 
				\begin{equation} 
				\label{optimalvalue} 
				V(y)=\sup\{\psi(y)\mid \psi\in\text{supp\,} V\}.
				\end{equation}
				\end{description}
				
				Ad $(a)$. In this case, $\text{supp\,} V=\emptyset$, i.e. $V\equiv-\infty$.
				
				Ad $(b)$. Assume that $V(y_{0})=\inf\limits_{x\in X}\sup\limits_{\psi\in\Psi}{\mathcal L}(x,\psi)$. By $(i)$, in view of Theorem 	\ref{optgen} and Proposition \ref{propoptvalue}, we have $V^{**}(y_{0})=\sup\limits_{\psi\in\Psi}\sup\limits_{x\in X}{\mathcal L}(x,\psi)=V(y_{0})$. By Theorem \ref{conju}, $V$ is $\Psi$-convex at $y_{0}$.
				\end{proof}
 
			The following theorem corresponds to the classical fact for convex functions (see i.e. \cite{Bonnans}, Theorem 2.142 \cite{bot}, Theorem 1.6). For abstract convex functions this theorem was proved in Theorem 5.2 of \cite{rubyang} (see also Proposition 2.1 of \cite{BurachikRubinov}).

		We say that an element $\psi_{0}\in\Psi$ is a solution to the Lagrangian dual problem \eqref{ldual} if 
		\begin{equation} 
		\label{dualsolution} 
		\sup_{\psi\in\Psi}\inf_{x\in X} \mathcal{L}(x,\psi)=\inf_{x\in X} \mathcal{L}(x,\psi_{0}).
		\end{equation}
		The function $q:\Psi\rightarrow\hat{\mathbb{R}}$ defined as
		$$
		q(\psi):=\inf_{x\in X} \mathcal{L}(x,\psi)
		$$ is called the {\em dual function}. With this notation, $\psi_{0}\in\Psi$ solves the Lagrangian dual \eqref{ldual} iff
		$$
		q(\psi_{0})=\sup_{\psi\in\Psi} q(\psi).
		$$

			\begin{theorem}
				\label{sub}
				The following statements hold:
				\begin{description}
\item{(i)}If   $\partial_{\Psi} V(y_0)\neq\emptyset$,  then  $val (L_{P}) =val (L_{D})$, and the 
solution set to the Lagrangian dual problem \eqref{ldual} coincides with $\partial_{\Psi} V(y_0)$.
\item (ii) If $val (L_{P}) =val (L_{D})$,  and both values are finite, then the (possibly empty) optimal solution
set of  \eqref{ldual} coincides with $\partial_{\Psi} V(y_0)$.
\end{description}
				\end{theorem}
			\begin{proof}
				By Proposition \ref{you}, the following equivalence holds
				$$
				\psi \in \partial_{\Psi}V(y_0) \  \ \Leftrightarrow\ \ \psi(y_0)-V^*(\psi)=V(y_0).
				$$
				Consequently,
				$$
				val(L_P)=V(y_0)=\psi(y_0)-V^*(\psi)\leq \sup_{\psi\in \Psi}\{\psi(y_0)-V^*(\psi)  \}=V^{**}(y_0)=val(L_D),
				$$
				which means that $val(L_P)=val(L_D)$ if and only if 	$\psi \in \partial_{\Psi}V(y_0) $. We have proved $(i)$ and $(ii)$.
			
				\end{proof}

\section{Duality for $\Phi_{conv}$-Lagrangian} 
\label{sectionconvex} 
Let $X$ be a Banach  space with the dual  $X^{*}$. In the present section we analyse the Lagrangian duality  for the $\Phi_{conv}$ Lagrangian, where the class $\Phi_{conv}$ defined in Example \ref{exampleone},
\begin{equation}
    \label{classconvex}
	\Phi_{conv}:=\{\varphi : X \rightarrow \mathbb{R},\ \varphi(x)= \left\langle \ell,x\right\rangle+c, \ \ x\in X,\  \ell\in X^{*}, \ c\in \mathbb{R}\},
\end{equation}
	and in the original  problem \eqref{problem}, the minimized function $f:X\rightarrow\bar{\mathbb{R}}$ is  a proper convex and lsc function of the form
\begin{equation} 
\label{funkcjaf}
f(x):=g(x)+h(x)
\end{equation}
where $g,h:X\rightarrow
\bar{\mathbb{R}}$ are convex proper lsc functions with $\text{dom\,} g\cap\text{dom\,} h\neq\emptyset$. 

For  problem \eqref{problem} with  $f$ given by \eqref{funkcjaf}, we consider the perturbation function $p:X\times X\rightarrow\bar{\mathbb{R}}$, 
\begin{equation} 
\label{convexperturbation}
p(x,y):=g(x)+h(x-y),\ \ \ p_{x}(\cdot):=p(x,\cdot):X\rightarrow\hat{\mathbb{R}}.
\end{equation}
Clearly, $y_{0}=0$, $0\in\text{dom\,} p_{x}(\cdot)\subset X$ for any $x\in\text{dom\,}f$ and $p$ is a proper function. 

Hence, we take $\Psi_{conv}:=\Phi_{conv}$. In the sequel,  we identify $X$ and $Y$ (i.e. $X=Y$) and  the class  $\Phi_{conv}$ with the Cartesian product $X^{*}\times\mathbb{R}$.

 For any $x\in X$ the conjugate $p_{x}^{*}:X^{*}\times\mathbb{R}\rightarrow\hat{\mathbb{R}}$, is given by the formula  
 \begin{equation} 
 \label{conjugacy}
 p_{x}(\psi):=p_{x}^{*}(y^{*},c)=\sup\limits_{y\in X}\{y^{*}(y)+c-g(x)-h(x-y)\}.
 \end{equation}
 This definition of conjugacy coincides with the classical one when $c=0$. Clearly, the conjugate function  \eqref{conjugacy}, is convex and lsc on $X^{*}\times\mathbb{R}$. The reason for defining conjugacy in the above (a bit unusual way) way is motivativated by the importance of constant $c\in\mathbb{R}$ in  the intersection property, Definition \ref {def_2}. Let us observe, that the appearance of the constant $c$  above (which amounts to considering affine functionals rather than linear) does not influence the values of the Lagrangian.
 
The Lagrangian ${\mathcal L}:X\times X^{*}\times\mathbb{R}\rightarrow\hat{\mathbb{R}}$ given by \eqref{genlag}   takes the form
\begin{equation} 
\label{convexlag1}
{\mathcal L}(x,\psi)={\mathcal L}(x,y^{*},c)\begin{array}[t]{l}
={\mathcal L}(x,y^{*})=-\sup\limits_{y\in X}\{\langle y^{*},y\rangle-g(x)-h(x-y)\}\\
= g(x)+\langle y^{*},x\rangle-\langle y^{*},x\rangle-\sup\limits_{y\in X}\{\langle y^{*},y\rangle-h(x-y)\}\\
=g(x)+\langle -y^{*},x\rangle-\sup\limits_{y\in X}\{\langle -y^{*},-y\rangle+\langle -y^{*},x\rangle-h(x-y)\}\\
=g(x)-\langle y^{*},x\rangle-h^{*}(-y^{*}).
\end{array}
\end{equation}
The Lagrangian ${\mathcal L}(\cdot, y^{*})$ is lsc and convex for any $y^{*}\in Y$  and ${\mathcal L}(x, \cdot)$ is usc and concave for any $x\in X$. 

In particular, when $h$ is the indicator function of a convex set
$A=\{x\in X\mid G(x)\in K\}$, where $G:X\rightarrow Z$, $Z$ is a Banach space, $K\subset Z$ is closed convex cone in $Z$, then
\begin{equation} 
\label{lagconstrained}
{\mathcal L}(x,y^{*},c)\begin{array}[t]{l}
=g(x)-\sup\limits_{y\in Y}\{\langle y^{*},y\rangle -\text{ind\,}_{A}(x-y)\}\ \ u:=x-y\\
=g(x)-\langle y^{*},x\rangle-\sup\limits_{u\in Y}\{\langle -y^{*},u\rangle -\text{ind\,}_{A}(u)\}\\
=g(x)-\langle y^{*},x\rangle-\sup\limits_{u\in A}\{\langle -y^{*},u\rangle\}\\
\end{array}
\end{equation}

For any $(y^{*},c)\in X^{*}\times \mathbb{R}^{*}$,
$$
\text{supp\,}{\mathcal L}(\cdot,y^{*},c):=\{(z^{*},d)\in  X^{*}\times\mathbb{R}\mid (z^{*},d)\le{\mathcal L}(\cdot,y^{*},c)\}.
$$
 Since $f$ is proper, by \eqref{proper}, there exist $(y^{*},c)\in X^{*}\times \mathbb{R}$ such that $\text{supp\,}{\mathcal L}(\cdot,y^{*},c)\neq\emptyset$.
 
 Suppose now that our original problem \eqref{problem} is the classical constrained optimization problem
 \begin{equation} 
 \label{inequality} 
 \min\limits_{x\in A} \ g(x),
 \end{equation}
 where $A:=\{x\in X\mid g_{x}\le 0,\ i=1,...,m\}$, $g_{i}:X\rightarrow\mathbb{R}$, $i=1,...,m$ and $p:X\times\mathbb{R}^{m}\rightarrow\hat{\mathbb{R}}$, 
 $$
 p(x,y):=\left\{\begin{array}{lll}
 f(x)&\text{when } & g_{i}(x)+y_{i}\le 0,\ i=1,...,m\\
 +\infty&&\text{otherwise} \\
 \end{array}\right.
 $$
 Hence, 
 $$
 f(x):=g(x)+\text{ind\,}_{K}(G(x)),
 $$
 where $K:=\{y=(y_{1},...,y_{m})\in\mathbb{R}^{m}\mid y_{i}\le 0,\ i=1,...,m\}$, $G(x):=(g_{1}(x),g_{2}(x),...,g_{m}(x))$ and the perturbation function is of the form
 $$
 p(x,y)=g(x)+\text{ind\,}_{K}(G(x)+y),\ \ p(y_{0})=p(0)=f(x),\ \ x\in X,\ \ y\in\mathbb{R}^{m}
 $$
 with $p^{*}_{x}(y^{*},c):=\sup\limits_{y\in Y}\{\langle y^{*},y\rangle+c-g(x)-\text{ind\,}_{K}(G(x)+y)\}$, $y^{*}\in\mathbb{R}^{m}$, $c\in\mathbb{R}$. Note that here $\Phi_{conv}$ is identified with $X^{*}\times\mathbb{R}$ and $\Psi$ is identified with $\mathbb{R}^{m}\times\mathbb{R}$.
 
 The Lagrangian \eqref{genlag}, ${\mathcal L}:X\times\mathbb{R}^{m}\times \mathbb{R}\rightarrow\hat{\mathbb{R}}$, takes the form
 $$
 {\mathcal L}(x,\psi)={\mathcal L}(x,y^{*},c)=-\sup\limits_{y\in\mathbb{R}^{m}}\{\langle y^{*},y\rangle-g(x)-\text{ind\,}_{K}(G(x)+y)\}.
 $$
We have
$$
 {\mathcal L}(x,y^{*},c)={\mathcal L}(x,y^{*})\begin{array}[t]{l}
 =-\sup\limits_{y\in\mathbb{R}^{m}}\{\langle y^{*},y\rangle-g(x)-\text{ind\,}_{K}(G(x)+y)\}\\
 =g(x)-\sup\limits_{y\in\mathbb{R}^{m}}\{\langle y^{*},y\rangle-\text{ind\,}_{K}(G(x)+y)\}\\
 =g(x)-\sup\limits_{y\in\mathbb{R}^{m}}\{\langle y^{*},y\rangle+\langle y^{*},G(x)\rangle-\langle y^{*},G(x)\rangle-\text{ind\,}_{K}(G(x)+y)\}\\
 =g(x)+\langle y^{*},G(x)\rangle-\sup\limits_{y\in\mathbb{R}^{m}}\{\langle y^{*},y\rangle+\langle y^{*},G(x)\rangle-\text{ind\,}_{K}(G(x)+y)\}\\
 =g(x)+\langle y^{*},G(x)\rangle-\sup\limits_{y\in\mathbb{R}^{m}}\{\langle y^{*},y\rangle+\langle y^{*},G(x)\rangle-\text{ind\,}_{K}(G(x)+y)\}\\
 \end{array}
 $$
 Observe that
 $$
 \sup\limits_{y\in\mathbb{R}^{m}}\{\langle y^{*},y\rangle+\langle y^{*},G(x)\rangle-\text{ind\,}_{K}(G(x)+y)\}=\left\{\begin{array}{lll}
 0&\text{when}&y^{*}_{i}\ge 0\\
 +\infty&&\text{otherwise}
 \end{array}\right.
 $$
 Consequently,
 $$
{\mathcal L}(x,y^{*})=\left\{\begin{array}{lll}
g(x)+\langle y^{*},G(x)\rangle&\text{when}& y^{*}\ge 0\\
-\infty&&\text{otherwise}
\end{array}\right.
$$
For problem \eqref{funkcjaf} with perturbation function \eqref{convexperturbation}, Theorem \ref{optgen} takes the following form.

\begin{theorem}
	\label{optconv}
	Let $X$ be a Banach space with the dual  $X^{*}$.  Let the  function ${\mathcal L}(x, x^{*},c):X\times X^{*}\times\mathbb{R}\rightarrow\hat{\mathbb{R}}$ be given by \eqref{convexlag1}.
	The following are equivalent:
	\begin{description}
\item[(i)]	for every $\alpha <\inf_{x\in X}\sup_{(x^{*},c)\in X^{*}\times\mathbb{R}}{\mathcal L}(x, x^{*},c)$ there exist $(x_{1}^{*},c_{1}),(x_{2}^{*},c_{2})\in X^{*}\times\mathbb{R}$ and  $ \varphi_1 \in\textnormal{supp} {\mathcal L}(\cdot,x_{1}^{*},c_{1})$  and $ \varphi_2 \in\textnormal{supp} {\mathcal L}(\cdot,x_{2}^{*},c_{2})$ such that 
$$
[\varphi_{1}<\alpha]\cap [\varphi_{2}<\alpha]=\emptyset;
$$
\item[(ii)]	$$\inf\limits_{x\in X} f(x)=\inf_{x\in X}\sup_{\psi\in\Psi}{\mathcal L}(x, \psi)=\sup_{\psi\in\Psi}\inf_{x\in X}{\mathcal L}(x, \psi).$$
\end{description}
\end{theorem}
\begin{proof} The proof follows from Proposition \ref{primal} and Theorem \ref{optgen}. Let the  function $p_{x}: X^{*}\times \mathbb{R}\rightarrow\hat{R}$, $p_{x}(y)=p(x,y)$, with $p$ is given by \eqref{convexperturbation} be $\Phi_{conv}$-convex on $Y$.
Let the  function ${\mathcal L}(\cdot, x^{*},c):X\rightarrow\hat{\mathbb{R}}$, given by \eqref{convexlag1},
	be $\Phi_{conv}$-convex on $X$   for any $(x^{*},c)\in X^{*}\times \mathbb{R}$.  
\end{proof}

\begin{proposition} 
Assume that 
$\inf\limits_{x\in X} f(x)$  finite. The following are equivalent.
\begin{description}
\item[(I)] Condition (i) of Theorem \ref{optconv} (the intersection property).
\item[(II)] For every $\varepsilon>0$ there exist $(y^{*}_{1},c_{1}),(y^{*}_{2},c_{2})\in Y^{*}\times\mathbb{R}$ and
$\varphi_{1}(x):=(z^{*}_{1},d_{1})\in\text{supp\,}{\mathcal L}(\cdot,y_{1}^{*},c_{1})$, $\varphi_{2}(x):=(z^{*}_{2},d_{2})\in\text{supp\,}{\mathcal L}(\cdot,y_{2}^{*},c_{2})$ such that
\begin{equation} 
\label{conditionii}
\begin{array}{l}
(a)\ \ t_{0}d_{1}+(1-t_{0})d_{2}\ge\inf_{x\in X} f(x)-\varepsilon,\\   (b)\ \ t_{0}z_{1}^{*}+(1-t_{0})z_{2}^{*}=0\ \ \ 
\end{array}
\end{equation} 
are satisfied for some $t\in[0,1]$.
\end{description}
\end{proposition}
\begin{proof}   Assume that (I) holds. Condition (i) of Theorem \ref{optconv} reads as follows: for every $\varepsilon>0$ one can find
$(y^{*}_{1},c_{1}),(y^{*}_{2},c_{2})\in Y^{*}\times\mathbb{R}$ and
\begin{equation} 
\label{conditioniin}
   (z^{*}_{1},d_{1})\in\text{supp\,}{\mathcal L}(\cdot,y_{1}^{*},c_{1}) \ \text{  and  }
(z^{*}_{2},d_{2})\in\text{supp\,}{\mathcal L}(\cdot,y_{2}^{*},c_{2})
\end{equation}
satisfying
\begin{equation} 
\label{intersection}
[x\mid \langle z^{*}_{1},x\rangle+d_{1}<\inf_{x\in X} f(x)-\varepsilon]\cap[x\mid \langle z^{*}_{2},x\rangle+d_{2}<\inf_{x\in X} f(x)-\varepsilon]=\emptyset.
\end{equation}
 Equivalently, (see Lemma 4.1 of \cite{Syga2018}) there exists $t_{0}\in[0,1]$ satisfying
\begin{equation}
\label{crucialemma}
\langle t_{0}z_{1}^{*}+(1-t_{0})z_{2}^{*},x\rangle+t_{0}d_{1}+(1-t_{0})d_{2}\ge\inf_{x\in X} f(x)-\varepsilon\ \ \text{for all  } x\in X,
\end{equation}
i.e. it must be
$$
 t_{0}z_{1}^{*}+(1-t_{0})z_{2}^{*}=0,
 $$
 and consequently
 \begin{equation}
\label{crucialemma1}
t_{0}d_{1}+(1-t_{0})d_{2}\ge\inf_{x\in X} f(x)-\varepsilon\ \ \text{for all  } x\in X.
\end{equation}
   Assume that (II) holds. Then the following inequality holds 
 \begin{equation}
\label{crucialemma2}
\langle t_{0}z_{1}^{*}+(1-t_{0})z_{2}^{*},x\rangle+t_{0}d_{1}+(1-t_{0})d_{2}\ge\inf_{x\in X} f(x)-\varepsilon\ \ \text{for all  } x\in X.
\end{equation}
Which is equivalent to 
$$
t_0\varphi_1(x)+(1-t_0)\varphi_2(x)\geq \inf_{x\in X} f(x)-\varepsilon
 \ \text{for all  } x\in X.
$$
where $(z^{*}_{1},d_{1})\in\text{supp\,}{\mathcal L}(\cdot,y_{1}^{*},c_{1}) \ \text{  and  }
(z^{*}_{2},d_{2})\in\text{supp\,}{\mathcal L}(\cdot,y_{2}^{*},c_{2})$.
By Lemma 4.1 of \cite{Syga2018} functions $\varphi_1$ and $\varphi_2$ have the intersection property at the level $\inf_{x\in X} f(x)-\varepsilon$. Since $\varepsilon>0$ was chosen arbitrarily, we get that functions $\varphi_1$ and $\varphi_2$ have the intersection property at every level $\alpha<\inf_{x\in X} f(x)$.

 \end{proof}
 
We close this section with a comparison between the intersection property and  two types of optimality conditions which are studied in the literature, namely, the so-called generalized interior point-condition and  closedness-type condition. These conditions are thoroughly researched in   \cite{Bot2012}.

The first example shows that the generalized interior point-condition fails, but the closedness-type one holds, as was proved in   \cite{Bot2012}. We show that the intersection property is fulfilled.
\begin{example}[\cite{Bot2012}, Example 21 and Example 23]
	Let $(X, \| \cdot \| )$ be a  real reflexive Banach space, $x_0^*\in X^*\setminus \{0 \}$ and the functions $f,g:X\rightarrow \bar{\mathbb{R}}$ be defined by $f(\cdot)=\textnormal{ind}_{\text{ker}\ x_0^*}(\cdot)$ and $g=\|\cdot\| + \textnormal{ind}_{\text{ker}\ x_0^*}(\cdot)$. We have
	$$\beta=\inf\limits_{x\in X}\sup\limits_{y^*\in X^*}{\mathcal L}(x, y^*)= \inf\limits_{x\in X }\{f(x)+g(x)  \}=\inf\limits_{x\in {\text{ker}\ x_0^*}}\| x\|=0$$
	and $g^*(y^*)=\textnormal{ind}_{B_*(0,1)+\mathbb{R} x_0^*}(y^*)$.
	Let $y_1^*=ax_{0}^*$, $a\in \mathbb{R}$. We have $-y^*_1\in B_*(0,1)+\mathbb{R} x_0^*$, and $g^*(-y_1^*)=0$, hence
	$$
	{\mathcal L}(x, y^*_1)=f(x)-\langle ax^*_0, x\rangle = \left \{ 
	\begin{matrix}
	0- \langle ax^*_0, x\rangle , \ \ \ \ \ \text{if } x\in \text{ker}\ x_0^* ,\cr
	+\infty - \langle ax^*_0, x\rangle \ \ \text{if } x\notin \text{ker}\ x_0^*
	\end{matrix}
	\right.= \left \{ 
	\begin{matrix}
	0, \ \ \ \ \ \text{if } x\in \text{ker}\ x_0^* ,\cr
	+\infty\ \ \  \text{if } x\notin \text{ker}\ x_0^*
	\end{matrix}
	\right.
	$$
	 Let $y_2^*=bx_{0}^*$, $b\in \mathbb{R}$, $b\neq a$. Let $\varphi_1,\varphi_2\in\Phi_{conv}$ and $\varphi_1=\varphi_2\equiv 0$, then $\varphi_1\in \text{supp} {\mathcal L}(\cdot, y^*_1)$ and $\varphi_2\in \text{supp} {\mathcal L}(\cdot, y^*_2)$. Functions $\varphi_1$ and $\varphi_2$ have the intersection property at the level $0$, hence at every level $\alpha<0$.
\end{example}
The next example shows that the generalized interior point-condition hold, but closedness-type one fails, as was shown in   \cite{Bot2012}. We show that the intersection property is fulfilled.
\begin{example}[\cite{Bot2012}, Example 25]
	Let $X=\ell^2(\mathbb{N})$ and let the set $C$, $S$ be such that
	$$
	C=\{(x_n)_{n\in \mathbb{N}}\in \ell^2\ : \ x_{2n-1}+x_{2n}=0 \ \forall n\in \mathbb{N}\},
	$$
	and
	$$
	S=\{(x_n)_{n\in \mathbb{N}}\in \ell^2\ : \ x_{2n}+x_{2n+1}=0 \ \forall n\in \mathbb{N}\},
	$$
	hence $S\cap C=\{0 \}$.
	Define the functions $f,g: \ell^2\rightarrow \bar{\mathbb{R}}$ by $f(\cdot)=\textnormal{ind}_C(\cdot)$, $g=\textnormal{ind}_{S}(\cdot)$, so
	$$
	\beta=\inf\limits_{x\in \ell^2}\sup\limits_{y^*\in\ell^2}{\mathcal L}(x, y^*)= \inf\limits_{x\in \in\ell^2 }\{f(x)+g(x)  \}=\{0\}.
	$$
	We have $g^*=\textnormal{ind}_{S^\perp}$, where
	$$
	S^\perp=\{(x_n)_{n\in \mathbb{N}}\in \ell^2\ : \ x_1=0, \  x_{2n}=x_{2n+1} \ \forall n\in \mathbb{N}\}.
	$$
	Let $y_1^*=(0,0,...,0,...)$, then we have
	$$
	{\mathcal L}(x, y^*_1)=f(x).
	$$
	Let $y^*_2$ be such that $-y^*_2\in S^\perp$, then
	$$
	{\cal L}(0,y^*_2)
	=f(0)-g^{*}( -y^*_2)-\langle y^*_2, 0\rangle= -g^{*}( -y^*_2)=0.
	$$
	Let $\varphi_1\in \text{supp}{\mathcal L}(\cdot, y^*_1)$ and $\varphi_2\equiv 0$, hence $\varphi_2\in \text{supp}{\mathcal L}(\cdot, y^*_2)$. Functions $\varphi_1$ and $\varphi_2$ have the intersection property at the level $0$, hence at every level $\alpha<0$.
\end{example}


 \section{ Duality for $\Phi_{lsc}$-convex Lagrangian.}
\label{section_3}

Let $X,Y$ be  Hilbert spaces. Let $\Phi_{lsc}$ and $\Psi_{lsc}$ be defined by \eqref{philsc} on $X$ and $Y$, respectively. 

In the present section we consider  problem \eqref{problem} of Section \ref{section_2} 
with the  perturbation function $p$ such that $p_{x}=p(x,\cdot):\Psi_{lsc}\rightarrow\bar{\mathbb{R}}$  is $\Psi_{lsc}$-convex for any $x\in X$ and with the Lagrangian ${\mathcal L}$, ${\mathcal L}(\cdot,\psi): X\rightarrow \bar{\mathbb{R}}$, which is  $\Phi_{lsc}$-convex for any $\psi\in\Psi$.


 
In the  considered case the  Lagrangian  ${\mathcal L}:X\times \Psi_{lsc}:\rightarrow \bar{\mathbb{R}}$  defined by \eqref{lag} takes the form 
\begin{equation}
\label{lag}
{\mathcal L}(x, \psi)=-a\|y_0\|^2-\langle v,y_0\rangle+c-p^{*}_{x}(\psi),
\end{equation}
where $\psi(y):= -a\|y\|^2+\langle v,y\rangle +c$ and 
$$
p_{x}^{*}(\psi)=\sup\limits_{y\in Y}\{ -a\|y\|^2+\langle v,y\rangle+c-p(x,y)\}=c+\sup\limits_{y\in Y}\{ -a\|y\|^2+\langle v,y\rangle-p(x,y)\}.
$$

\begin{remark}
    \label{remark-on_c}
Observe that the right-hand side of \eqref{lag} is independent of $c$, i.e. for all functions $\psi$ which differs only be $c$ (i.e. have the same $a$ and $v$) the right-hand side of \eqref{lag} is the same.  In view of this
\begin{equation}
\label{lag1}
{\mathcal L}(x, \psi)={\mathcal L}(x, a,v) =-a\|y_0\|^2+\langle v,y_0\rangle-p^{*}_{x}(a,v),
\end{equation}
where 
\begin{equation}
\label{psiconj}
p_{x}^{*}(a,v)=\sup\limits_{y\in Y}\{ -a\|y\|^2+\langle v,y\rangle-p(x,y)\},
\end{equation}
i.e. the Lagrangian ${\mathcal L}$ can be equivalently regarded as a function defined on $X\times \mathbb{R}_{+}\times Y^{*}$.
\end{remark}

According to Proposition \ref{primal}, problem \eqref{problem} is equivalent to the Lagrangian primal problem 
\begin{equation} 
\label{lagrange_primal}
\inf\limits_{x]in X}\ f(x)=\inf\limits_{x\in X}\sup\limits_{\psi\in \Psi_{lsc}} {\mathcal L}(x,\psi)
\end{equation}
provided $p(x,\cdot)$ is $\Psi_{lsc}$-convex for any $x\in X$.

The following theorem is based on Theorem \ref{optgen}.

\begin{theorem}
	\label{opt}
		Let $X,Y$ be Hilbert spaces. Let $p(x,\cdot)$ be $\Psi_{lsc}$-convex function at $y_{0}\in Y$ for all $x\in X$. Let ${\mathcal L}:X\times  \Psi_{lsc}\rightarrow\hat{\mathbb{R}}$  be the Lagrangian defined by \eqref{lag}, i.e.
	$$
	{\mathcal L}(x, \psi)=-a\|y_0\|^2+\langle v,y_0\rangle +c-p^{*}_{x}(\psi).
	$$
	
	Assume  that for any $\psi\in\Psi_{lsc}$ the  function ${\mathcal L}(\cdot,\psi):X\rightarrow\hat{\mathbb{R}}$ 
	is $\Phi_{lsc}$-convex on $X$.  The following are equivalent:
		\begin{description}
		\item[(i)]	for every $\alpha <\inf\limits_{x\in X}\sup\limits_{\psi\in\Psi_{lsc}} {\mathcal L}(x, \psi)$ there exist $\psi_1, \psi_2\in  \Psi$ and  $ \varphi_1 \in\text{supp} {\mathcal L}(\cdot,\psi_1)$  and $ \varphi_2 \in\text{supp} {\mathcal L}(\cdot,\psi_2)$ such that 
		$$			[\varphi_1<\alpha]\cap [ \varphi_2<\alpha]=\emptyset
			$$
\item[(ii)]
	$$\inf\limits_{x\in X} f(x)=\sup_{\psi\in\Psi_{lsc}}\inf_{x\in X}{\mathcal L}(x, \psi)=\inf_{x\in X}\sup_{\psi\in\Psi_{lsc}}{\mathcal L}(x, \psi).$$
\end{description}
\end{theorem}

\begin{proof}
	Follows  from formula \eqref{eq-n1}, Proposition \ref{primal} and Theorem \ref{optgen}.
\end{proof}

Let $\beta:=\inf\limits_{x\in X}\sup\limits_{\psi\in\Psi_{lsc}} {\mathcal L}(x,\psi)$.
Since the intersection property follows from the zero subgradient condition (Proposition \ref{subzero}),  we can prove the following sufficient conditions for zero duality gap.
\begin{theorem}
	\label{new_min_max_2}
	Let $X,Y$ be Hilbert spaces. Let $p(x,\cdot)$ be $\Psi_{lsc}$-convex function at $y_{0}\in Y$ for all $x\in X$. Let ${\mathcal L}:X\times  \Psi_{lsc}\rightarrow\hat{\mathbb{R}}$ be the Lagrangian defined by \eqref{lag}, i.e.
	$$
	{\mathcal L}(x, \psi)=-a\|y_0\|^2+\langle v,y_0\rangle +c-p^{*}_{x}(\psi).
	$$
	 Assume  that for any $\psi\in\Psi_{lsc}$ the  function ${\mathcal L}(\cdot,\psi):X\rightarrow\hat{\mathbb{R}}$ 
	is $\Phi_{lsc}$-convex on $X$.  

If	there exist $\psi_1,\psi_2\in  \Psi_{lsc}$ and $ \bar {x} \in \text{dom\,} f$,
	 $ \bar {x} \in [{\mathcal L}(\cdot,\psi_1)\geq \beta]\cap [{\mathcal L}(\cdot,\psi_2)\geq \beta]$ such that 
	 $$
	 0\in \text{co}(\partial_{lsc}{\mathcal L}(\bar{x},\psi_1)\cup \partial_{lsc}{\mathcal L}(\bar{x},\psi_2)),
	 $$
	  then
	$$\inf\limits_{x\in X} f(x)=\sup_{\psi\in\Psi_{lsc}}\inf_{x\in X}{\mathcal L}(x, \psi)=\inf_{x\in X}\sup_{\psi\in\Psi_{lsc}}{\mathcal L}(x, \psi).$$
\end{theorem}
\begin{proof}
Since $\text{dom\,}{\mathcal L}(\cdot,\psi)\supset\text{dom\,}f$, for all $\psi\in \Psi_{lsc}$, proof
follows immediately from Proposition \ref{subzero}, Proposition \ref{primal} and Theorem \ref{new_min_max}. 
 \end{proof}

We say that a function $f:Y\rightarrow \bar{\mathbb{R}}$ is paraconvex (in the literature paraconvex functions are known also under the name weakly convex \cite{Vial} and semiconvex \cite{cannarsa}) on $Y$ if there exists $c>0$ such that $f+c\|\cdot\|^2$ is convex. Equivalently, a function $f:Y\rightarrow\bar{\mathbb{R}}$ is called paraconvex on $X$ if there exists $C>0$ such that for all $x,y\in Y$  and $t\in [0,1]$ the following inequality holds
 	\begin{equation}
 	\label{para}
 	f(tx+(1-t)y)\leq tf(x)+(1-t)f(y)+C \|x-y\|^2,
 	\end{equation}
see e.g. \cite{Rolewiczpara}. It was shown in \cite{jurani96} and \cite{rolewicz2000} that \eqref{para} is equivalent to
	\begin{equation}
 	\label{para1}
 	f(tx+(1-t)y)\leq tf(x)+(1-t)f(y)+Ct(1-t) \|x-y\|^2.
 	\end{equation}
It was shown in \cite{sygamoor}, Proposition 3 that every  lower semicontinuous  paraconvex function is $\Phi_{lsc}$-convex. Moreover, by Proposition 5 of \cite{sygamoor}, if $f$ is proper lsc and paraconvex on $Y$, then for every $y\in \text{int\,dom} (f)$ the subdifferential $\partial_{lsc}f(y)$ is nonempty. 

Proposition 5 of \cite{sygamoor} allows to obtain the strong duality theorem for  the pair of dual problems problem \ref{problem} with paraconvex optimal value function $V$ defined by \eqref{opvalue}, i.e.
$$
		    	V(y):=\inf_{x\in X}p(x,y).
		    	$$

As usual, we say that  that the {\em strong duality} holds  for dual problems \ref{lprob} and \ref{ldual} when  the zero duality gap holds and the dual problem  \ref{ldual} is solvable.

The following strong duality result holds for problems   with the paraconvex (weakly convex) optimal value function $V$
and $\text{int\,}dom V\neq\emptyset$.

\begin{theorem}
	\label{par}
	Let $X$ and $Y$ be Hilbert spaces. 
	Assume that the perturbation function $p(x,\cdot)$ (as a function of $y$) is $\Psi_{lsc}$-convex at $y_0$ for all $x\in X$.
	
	If the optimal value function $V(y)$  of the problem $(P)$ is proper, lsc, and paraconvex on $Y$  and $y_0\in \text{int\,dom} (V)$, then 
	$$
	\inf_{x\in X} f(x)=\sup_{\psi\in \Psi_{lsc}}\inf_{x\in X}{\mathcal L}(x, \psi)=\inf_{x\in X}\sup_{\psi\in \Psi_{lsc}}{\mathcal L}(x, \psi),
	$$
	i.e. $\inf_{x\in X} f(x)=val(L_P)=val(L_D)$ and the 
	solution set of the dual problem $\eqref{ldual}$ is nonempty and coincides with $\partial_{lsc} V(y_0)$.
\end{theorem}
\begin{proof}
 By Proposition \ref{primal}, 
 $$
 	\inf_{x\in X} f(x)=\inf_{x\in X}\sup_{\psi\in \Psi_{lsc}}{\mathcal L}(x, \psi).
 $$
 By paraconvexity of the function $V$ and by  Proposition 5 of \cite{sygamoor}, we get $\partial_{lsc}V(y_0)\neq\emptyset$. Hence, by Theorem \ref{sub}(i) we get  
 $$
 \inf_{x\in X}\sup_{\psi\in \Psi_{lsc}}{\mathcal L}(x, \psi)=\sup_{\psi\in \Psi_{lsc}}\inf_{x\in X}{\mathcal L}(x, \psi)
 $$
 and the 
	solution set to the dual problem \ref{ldual} is nonempty and coincides with $\partial_{lsc} V(y_0)$.
\end{proof}
Let us note that in the above theorem it is possible to replace the $\textnormal{int}\,\textnormal{dom} (V)$ by  the so called quasi-relative interior  $\textnormal{qri}\,\textnormal{dom} (V)$ see i.e Corollary 9 in \cite{zalinescu2015}.

 Theorem \ref{opt} and Theorem \ref{new_min_max_2} refer to generic perturbation function $p$. For some particular choices of  $p$ the intersection property is automatically satisfied. One  such choice is considered below.

 \section{Special case}
 In the present section we investigate Lagrangian duality for constrained optimization problems within the framework of $\Phi_{lsc}$-convexity.
 
 Let $X$ be a Hilbert space.
 Consider  the constrained optimization problem of the form
	\begin{equation}
	\label{problem211}
	\text{Min}\ \ \ \ f(x) \ \ \ \ \ \ x\in A(y_{0}),
	\end{equation}
	where,  as previously, $f:X\rightarrow\bar{\mathbb{R}}$ is a proper function in the sense that $\text{dom\,} f\neq \emptyset$ and $f>-\infty$, and $A: Y\rightrightarrows X$ is a set-valued mapping.
	The corresponding family of parametrized/perturbed problems $(P_y)$
	\begin{equation}
	\label{problem221}
	\text{Min}\ \ \ \ p(x,y) \ \ \ \ \ \ x\in X
	\end{equation}
	is based on the perturbation function $p:X\times \mathbb{R}^m\rightarrow \mathbb{R}\cup \{+\infty \}$  defined as
 (see \cite{rocka})  
		\begin{equation} 
		\label{funkcjap}
		p(x,y)=\left \{ 
	\begin{matrix}
	f(x) , \ \ \ \ \ \ \  \ x\in A(y)\cr
	+\infty,  \ \ \ \ \ x\notin A(y)
	\end{matrix}
	\right.,
	\end{equation}
	Problem \eqref{problem211} casts into general form \eqref{problem} with the minimized function $\tilde{f}(x):=f(x)+\text{ind}_{A(y_{0})}(\cdot)$, where $\text{ind}_{A}(\cdot)$ is the indicator function of set $A$. We adopt the convention that $\text{ind\,}_{A}\equiv+\infty$ whenever $A=\emptyset$.
Assume that 

$\beta=\inf\limits_{x\in X}\sup\limits_{\psi\in \Psi_{lsc}}{\mathcal L}(x, \psi)$ is finite.
	\begin{theorem}
	\label{opt1}
	Let $X, Y$ be Hilbert spaces and let $\inf\limits_{x\in X} f(x)$ be finite. Let ${\mathcal L}:X\times  \Psi_{lsc}\rightarrow\hat{\mathbb{R}}$  be the Lagrangian defined by \eqref{lag} with $y_{0}=0$,
	\begin{equation}
\label{laglscy0}
{\mathcal L}(x, \psi)=c-p^{*}_{x}(\psi),
\end{equation}
where, as previously, 
$\psi(y):= -a\|y\|^2+\langle v,y\rangle +c$ and 
$$
p_{x}^{*}(\psi)=\sup\limits_{y\in Y}\{ -a\|y\|^2+\langle v,y\rangle+c-p(x,y)\}=c+\sup\limits_{y\in Y}\{ -a\|y\|^2+\langle v,y\rangle-p(x,y)\}.
$$
	
	If, for any $\psi\in\Psi_{lsc}$, the  function  ${\mathcal L}(\cdot,\psi):X\rightarrow\hat{\mathbb{R}}$ 
	is $\Phi_{lsc}$-convex on $X$ and the perturbation function
  $p(x,\cdot)$ is $\Psi_{lsc}$-convex at $y_0$ for all $x\in X$, then
	$$
	\inf_{x\in X} f(x)=\sup_{\psi\in \Psi_{lsc}}\inf_{x\in X}{\mathcal L}(x, \psi)=\inf_{x\in X}\sup_{\psi\in \Psi_{lsc}}{\mathcal L}(x, \psi).
	$$
\end{theorem}
\begin{proof}
The first equality follows from 
Proposition \ref{primal}
applied to function 
$p_{x}$,  $p_{x}(y)=p(x,y)$, $x\in X$, $y\in Y$ given by \eqref{funkcjap}.

	In view of  Theorem \ref{opt}, to prove the second equality, we need to show that for every\\ $\alpha <\inf\limits_{x\in X}\sup\limits_{\psi\in\Psi_{lsc}} {\mathcal L}(x, \psi)$ there exist $\psi_1,\psi_2\in \Psi_{lsc}$ and  $ \varphi_1 \in\text{supp} {\mathcal L}(\cdot,\psi_1)$  and $ \varphi_2 \in\text{supp} {\mathcal L}(\cdot, \psi_2)$ such that
	$$
			[\varphi_1<\alpha]\cap [ \varphi_2<\alpha]=\emptyset.
			$$
			Let $\bar{\psi}\in \Psi_{lsc}$,  $\bar{\psi}(y)=-a\|y\|^2+c$ where $a>0$ and $c\in\mathbb{R}$. Let $x\in A_{y}$. By \eqref{lag}, 
			
			$$
			{\cal L}(x,\bar{\psi})\begin{array}[t]{l}
			=c-\sup_{y\in Y}\{	-a\|y\|^{2}+c-p(x,y)\}\\
				=-\sup_{y\in Y}\{	-a\|y\|^{2}-f(x)\}\\
					=f(x)-\sup_{y\in Y}\{	-a\|y\|^{2}\}\\
		= f(x) \geq \inf(P)=\beta.
		\end{array}
			$$
			If $x\notin A(y)$, then ${\cal L}(x,\bar{\psi})=+\infty$. Hence, for all $x\in X$
			$$
			{\cal L}(x,\bar{\psi})	\geq f(x) \geq \inf(P)=\beta.
			$$
			The function $\varphi_1\equiv \beta $ belongs to the support set of ${\cal L}(\cdot,\bar{\psi})$. Let $\varphi_2$ be any function from the set $\text{supp} {\cal L}(\cdot,\psi_2)$, for $\psi_2\in\Psi_{lsc}$, of the form $\psi_2(y)=-a_2\|y\|^2+c_2$, where $a_2>0$ and $c_2\in\mathbb{R}$,  then $\varphi_1$ and $\varphi_2$ have the intersection property at the level $\beta$, hence $\varphi_1$ and $\varphi_2$ have the intersection property at every level $\alpha< \beta$.
\end{proof}

Let $g_i:X\rightarrow \mathbb{R}$, $i=1,...,m$ be given functions. 
Let $A: \mathbb{R}^m\rightrightarrows X$ be given as
\begin{equation} 
\label{constrset}
A(y):=\{x\in X:\ g_i(x)\leq y_i, i=1,..,m\},
\end{equation}
$y_{0}=0$, $A:=A(0)$.
Consider the problem
\begin{equation} 
\label{AP}
\tag{AP}
\text{Min}\ \ \ \ f(x),\ \ \ \ x\in A.
\end{equation}

Problem \eqref{AP} can be equivalently rewritten as 
\begin{equation} 
\label{AP1}
\tag{AP1}
\text{Min}_{x\in X} \tilde{f}(x)
\end{equation}
where $\tilde{f}(x):=f(x)+\text{ind\, }_{A}(x)$ and $\text{ind\, }_{A}(\cdot)$ is the indicator function of the set $A$.

According to \eqref{funkcjap}, the perturbation  function $p:X\times \mathbb{R}^m\rightarrow \mathbb{R}\cup \{+\infty \}$ takes the form (see e.g. \cite{rocka}) as 
\begin{equation}
\label{p1}
	p(x,y)=\left \{ 
\begin{matrix}
f(x) , \ \ \ \ \ \ \  \ x\in A(y)\cr
+\infty,  \ \ \ \ \ x\notin A(y)
\end{matrix}
\right.,
\end{equation}
where $A(y)$ are given by \eqref{constrset}.
Let $\psi(y)=-a\|y\|^2+\langle v, y\rangle +c$. According to \eqref{lag}, the Lagrangian for  problem $(AP)$,  ${\mathcal L}:X\times \Psi_{lsc} \rightarrow\bar{\mathbb{R}}$, with $y_0=0$, is of the form
$$
{\mathcal L}(x, \psi)\begin{array}[t]{l}
=-p^{*}_{x}(\psi)=\\
=-\sup_{y\in Y}\{-a\|y\|^{2}+\langle v,y\rangle -p(x,y)\}\\
=-\sup_{y\in Y,\ y\ge g(x)}\{-a\|y\|^{2}+\langle v,y\rangle -f(x)\}\\
=f(x)+\inf_{y\in Y,\ y\ge g(x)}\{a\|y\|^{2}-\langle v,y\rangle \}.
\end{array}
$$
It is easy to check that
$$
\text{arg}\inf_{y\in Y,\ y\ge g(x)}\{a\|y\|^{2}-\langle v,y\rangle \} =
(y_{i})_{i=1}^{m}=(\max\{g_{i}(x),\frac{v_{i}}{2a}\})_{i=1}^{m}
$$
where $v=(v_1,...,v_m)$ and hence 
\begin{equation}
\label{lag2}
{\mathcal L}(x,\psi)=f(x)+\sum_{i=1}^{m}[-v_i\max\{g_i(x),\frac{v_i}{2a}\}+a(\max\{g_i(x),\frac{v_i}{2a}\})^2].
\end{equation}

 Similar  Lagrangian is defined in \cite{balder} and \cite{rocka}. As observed above,  we can write ${\mathcal L}(x,\psi)={\mathcal L}(x,a,v)$, for $(a,v)\in\mathbb{R}_{+}\times\mathbb{R}^{m}$. 

For any $x\in X$, if $\max\{g_{i}(x),\frac{v_{i}}{2a}\}=\frac{v_{i}}{2a}$ for $i=1,...,m$, then
$$
\sup_{(a,v)\in\mathbb{R}^{+}\times\mathbb{R}^{m}}{\mathcal L}(x,a,v)=f(x).
$$
Moreover,
 \begin{equation} 
 \label{lag2bb} 
 \sup_{(a,v)\in\mathbb{R}^{+}\times\mathbb{R}^{m}}{\mathcal L}(x,a,v)
  =\left\{\begin{array}{ll}
 f(x)& \text{  whenever  } g_{i}(x)\le 0 \ \& \ v_{i}\ge 0\ \forall\ i=1,...m\\
 +\infty&\text{otherwise}\\
 \end{array}\right.
 \end{equation}

Let $\beta=\inf\limits_{x\in X}\sup\limits_{\psi\in \Psi_{lsc} }{\mathcal L}(x, \psi)$ be finite number. By \cite{rocka} we have the following equality
$$
\inf(P)=	\inf_{x\in X}\sup_{\psi \in \Psi_{lsc} }{\mathcal L}(x, \psi).
$$

\begin{corollary}
	\label{phi-lag}
	Consider the problem \eqref{AP} with the perturbation function given by \eqref{p1}. Assume that the functions $f,g_{i}:X\rightarrow \mathbb{R}$ are $\Phi_{lsc}$-convex on $X$. Then

	$$
	\inf_{x\in X}\ \tilde{f}(x)=\inf_{x\in X}\sup_{(a,v) \in \mathbb{R}_+\times \mathbb{R}^m_{+} }{\mathcal L}(x, a,v)=
	\sup_{(a,v) \in \mathbb{R}_+\times \mathbb{R}^m_{+} }	\inf_{x\in X} {\mathcal L}(x, a,v),
	$$
	
	with the Lagrangian ${\mathcal L}$  given by \eqref{lag2}.
\end{corollary}
 
\begin{proof} 
By \eqref{lag2bb},
$$
\sup_{(a,v) \in \mathbb{R}_+\times \mathbb{R}^m }{\mathcal L}(x, a,v)=\sup_{(a,v) \in \mathbb{R}_+\times \mathbb{R}^m_{+} }{\mathcal L}(x, a,v).
$$
 In view of Theorem \ref{opt1} we need to show that the function $p_{x}$ is $\Psi_{lsc}$ on $Y$ and the Lagrangian ${\mathcal L}(\cdot,a,v)$ given by \eqref{lag2} is $\Phi_{lsc}$-convex for all $(a,v)\in  \mathbb{R}_{+} \times \mathbb{R}^m$. 
 
 To show that the function $p_{x}:\mathbb{R}^{m}\rightarrow \bar{\mathbb{R}}$ is $\Psi_{lsc}$ on $\mathbb{R}^{m}$ observe that for any $y=(y_{i})_{i=1}^{m}\in\mathbb{R}^{m}$ and $x\in X$
 $$
 p_{x}(y)=f(x) +\text{ind\,}_{B}(y),
 $$
 where for any fixed $x\in X$ we put $B:=\{y\in \mathbb{R}^{m}\mid g(x)\le y\}$, where $g(x):=(g_{i}(x))_{i=1}^{m}$. Since $B$ is a convex set the function $p_{x}$ is convex for any $x\in X$.
 

It is enough to observe that the function $\max\{h(x), const\}$ is $\Phi_{lsc}$-convex whenever $h$ is $\Phi_{lsc}$.
	\end{proof}
	
	By taking $u:=-v$, the formula \eqref{lag2} can be equivalently rewritten as
\begin{equation}
\label{lag2a}
{\mathcal L}(x,a,u)=f(x)+\sum_{i=1}^{m}[u_i\max\{g_i(x),\frac{-u_i}{2a}\}+a(\max\{g_i(x),\frac{-u_i}{2a}\})^2]
\end{equation}
which coincides with the formula 1.3 of \cite{rocka} (see also Example 2b of \cite{balder}).

	In view of Corollary \ref{phi-lag}, zero duality gap holds for the  Lagrangian of \cite{rocka}  for  $\Phi_{lsc}$-convex problem  (P), where the function $f, g_{i}$, $i=1,...,m$ are $\Phi_{lsc}$-convex.
	By the proof of Corollary \ref{phi-lag}, the perturbation function
	$p(x,\cdot)=p_{x}:\mathbb{R}^{m}\rightarrow \bar{\mathbb{R}}$ defined by \eqref{p1} is convex.  Hence, by Theorem \ref{par} we obtain the following fact.
	
 \begin{corollary} 
 \label{strongduality1} 
Let $X$  be a Hilbert space and $Y=\mathbb{R}^{m}$.
	
	If the optimal value function $V(y)$  of the problem $(P)$ is proper, lsc, and paraconvex on $Y$  and $y_0\in \text{int\,dom} (V)$, then 
	$$
	\inf_{x\in X} f(x)=\sup_{(a,v) \in \mathbb{R}_+\times Y^*}\inf_{x\in X}{\mathcal L}(x, a,v)=\inf_{x\in X}\sup_{(a,v) \in \mathbb{R}_+\times Y^*}{\mathcal L}(x, a,v),
	$$
	i.e. $val(L_P)=val(L_D)$ and the 
	solution set of the dual problem $\eqref{ldual}$ is nonempty and coincides with $\partial_{lsc} V(y_0)$.
 \end{corollary}
 \begin{proof}
Follows directly from Corollary  \ref{phi-lag} and Theorem \ref{par}. 
\end{proof}


\section{Conclusions}
In Theorem 4 and Theorem 6 we provide sufficient and necessary conditions for zero duality gap for pairs of dual optimization problems involving $\Phi$-convex  and $\Phi_{lsc}$-convex functions. In particular, our results apply to optimization problems where the considered Lagrangian, and the function $p(\cdot,\cdot)$ are paraconvex, or prox-bounded, or DC functions.

Let us observe that Theorem 2,  provides considerable flexibility in choosing Lagrange function ${\mathcal L}$. 
Sufficient and necessary conditions of  Theorem 4 and Theorem 6 are  based on the intersection property (Definition \ref{def_2}), which, in contrast to many existing in the literature conditions, is of purely algebraic character. 


\bibliographystyle{spmpsci}
\bibliography{bibn}
\end{document}